\documentclass[11pt,twoside]{amsart}
\usepackage{amsmath}
\usepackage{amssymb}
\usepackage{amsfonts}
\usepackage[all]{xy}
\usepackage[colorlinks=true,linkcolor=red,urlcolor=blue,citecolor=blue]{hyperref}

\newtheorem{theorem}{Theorem}[section]
\newtheorem{lemma}[theorem]{Lemma}
\newtheorem{proposition}[theorem]{Proposition}
\newtheorem{corollary}[theorem]{Corollary}
\theoremstyle{definition}
\newtheorem{definition}[theorem]{Definition}

\newcommand{\C}{\mathbb C}

\newcommand{\mC}{{\mathbb C}}

\newcommand{\mE}{{\mathbb E}}

\newcommand{\mQ}{\mathbb Q}

\newcommand{\mZ}{{\mathbb Z}}

\newcommand{\Gg}{\gamma}

\newcommand{\bs}{\sigma}

\newcommand{\kk}{\kappa}

\newcommand{\mcG}{\mathcal G}
\newcommand{\mcH}{\mathcal H}

\newcommand{\mcK}{\mathcal K}

\newcommand{\mcR}{\mathcal R}

\newcommand{\mcV}{\mathcal V}

\newcommand{\mc}{\mathcal}

\newcommand{\QQ}{\mathbb{Q}}
\newcommand{\ZZ}{\mathbb{Z}}

\newcommand{\RR}{\mathbb{R}}

\newcommand{\ti}{\tilde}

\newcommand{\Irr}{\mathrm{Irr}}
\newcommand{\Rep}{\mathrm{Rep}}

\newcommand{\Aut}{\mathrm{Aut}}
\newcommand{\Hom}{\mathrm{Hom}}
\newcommand{\Gal}{\mathrm{Gal}}
\newcommand{\OQ}{{\overline{\mathbb Q}}}
\newcommand{\nr}{\mathrm{nr}}
\newcommand{\der}{\mathrm{der}}

\newcommand\fdeg{\mathrm{fdeg}}

\usepackage{xcolor}
\usepackage{amscd}

\counterwithin*{equation}{section}

\begin{document}

\title[Rationality properties of representations of $p$-adic groups]
{Rationality properties of complex representations of reductive $p$-adic groups}
\date{\today}
\subjclass[2010]{20G25, 22E50}
\author{David Kazhdan}
\address{Einstein Institute of Mathematics\\
The Hebrew University of Jerusalem\\
Givat Ram, Jerusalem, 9190401, Israel}
\email{kazhdan@math.huji.ac.il}
\author{Maarten Solleveld}
\address{IMAPP, Radboud Universiteit Nijmegen, Heyendaalseweg 135,
6525AJ Nij\-megen, the Netherlands}
\email{m.solleveld@science.ru.nl}
\author{Yakov Varshavsky}
\address{Einstein Institute of Mathematics\\
The Hebrew University of Jerusalem\\
Givat Ram, Jerusalem, 9190401, Israel}
\email{yakov.varshavsky@mail.huji.ac.il}
\maketitle

%\vspace{-5mm}

\begin{abstract}
For a reductive group $G$ over a non-archimedean local field, we compare smooth 
representations over $\C$ with smooth representations over $\OQ$. We 
show that an elliptic $G$-representation (in the sense of Arthur) can be realized over 
$\OQ$ if and only if its central character takes values in $\OQ$. That
applies in particular to all essentially square-integrable $G$-representations.

We also study the action of the automorphism group of $\C/\mQ$ on complex 
$G$-representations. We prove that the sets of essentially square-integrable 
representations and of elliptic representations are stable under Gal$(\C/\mQ)$. 
\end{abstract}
\vspace{2mm}

\tableofcontents

\section{Introduction}
The paper \cite{FS} develops the theory of smooth representations of a reductive group $G$ over a
non-archimedean local field $F$ on $\OQ _l$-vector spaces. Choosing an isomorphism  between
$\mC$ and $\OQ _l$, one obtains results about complex representations of $G$. It is natural to
ask about independence of the choice of such an isomorphism. Since any two isomorphisms  between
$\mC$ and $\OQ _l$ differ by a composition with an element of the automorphism group
Gal$(\mC /\mQ)$, this question is equivalent to understanding of
Gal$(\mC /\mQ)$ on the category of smooth complex representations of $G$. In this paper we show
that a large part of the familiar structure of this category is invariant under Gal$(\mC /\mQ)$.

This can be considered to be a counterpart of the recent result of Scholze
\cite{Sc} asserting that the Fargues–-Scholze correspondence can be done motivically
and thus is independent of the $\ell$ in $\OQ_\ell$.

\subsection{Notations} 
\begin{enumerate}
\item $\OQ$ is the algebraic closure of $\mQ$ in $\C$.

\item  $F$ is a non-archimedean local field whose residue field has characte\-ristic
$p$ and cardinality $q_F$.

\item $G$ is the group of $F$-points of connected reductive $F$-group $\mcG$
and $Z(G)$ is the centre of $G$.

%\item  $G^{rs}\subset G$ is the open dense subset of $g\in G$  whose  centralizer $Z_G(g)$ is a torus.

%\item  $\mC (G^{rs}) $ the space of complex valued $Ad (G)$-invariant locally constant
%functions on $G^{rs}$.

%\item  $\OQ (G^{rs})\subset \mC (G^{rs}) $ the subspace of $ \OQ $-valued functions.

\item A representation of $G$ on a complex vector space will be identified with a representation of
the group algebra $\C G$.

\item $\Irr (\mC G)$ is the set of isomorphism classes of smooth complex irreducible
representations of $G$.

\item For $\pi \in \Irr (\mC G) $ we denote by $cc(\pi) : Z(G)\to \C ^\times $ the character
 such that $\pi (z)= cc(\pi)(z)Id, z\in Z(G)$.

\item $\Irr' (\mC G) = \{ \pi \in \Irr (\mC G) \mid \text{Im} ( cc(\pi)) \subset  \OQ^\times \}$.

\item By default, our representations will have complex coefficients. So if it says $G$ or $\C G$,
then the coefficients are in $\C$, while $\OQ G$ means that the coefficients of the
representations must be in $\OQ$.

\item $\mcR_G$ is the category of smooth finite length complex representations of the group
$G$. The variation $\mcR'_G$ is given by restricting to representations all whose
irreducible subquotients lie in $\Irr' (\mC G) $. When $G$ is semisimple,
$\mcR'_G$ coincides with $\mcR_G$.

\item $\Irr (\OQ G)$ is the set of equivalence classes
of smooth irreducible $\OQ$-representations of $G$.

\item $\mcR _{\OQ G}$ is the category of smooth finite length  $ \OQ $-representations of $G$.

\item  $ K( \mcR _G), K(\mcR'_G) $ and $ K( \mcR _{\OQ G}) $ are the $K$-groups of the categories
$ \mcR_G, \mcR'_G$ and  $\mcR _{\OQ G}$.
\item   $\mcK _G := K( \mcR _G) \otimes _\mZ \OQ ,\;  \mcK' _G := K( \mcR' _G) \otimes _\mZ \OQ ,\;
\mcK _{\OQ G} := K( \mcR_{\OQ G}) \otimes _\mZ \OQ $.

%\item  $\nu : \mcK _G \to  \mC (G^{rs}) $ is the map  which associates with  $V\in \mcK _G $
%the restriction of the Harish-Chandra character of $V$ onto  $G^{rs}$.

%\item $ \mcV_G := \nu (\mcK _G),  \mcV'_G := \nu (\mcK _G), \mcV _{\OQ G} := \nu (\mcK _{\OQ G})$.

\item $\mcK_{G,\text{temp}}$ is the subspace of $\mcK_G$ spanned by the tempered
$G$-repre\-sen\-ta\-tions. 

\item The definition of tempered \cite{Wal} applies also to
$\OQ G$-representations, so we can define $\mcK_{\OQ G,\text{temp}} \subset \mcK_{\OQ G}$ 
and $\mcK'_{G,\text{temp}} \subset \mcK'_G$ analogously.

\item J. Arthur (see \cite{Ar,Herb}) defined a $\OQ$-linear subspace
$\mcK_{G,\text{temp}} (ar) \subset \mcK _{G,\text{temp}} $ of elliptic tempered characters.

%\item $\mcK_{G,\text{temp}} (ar) \subset \mcK_G$ is the $\OQ$-linear subspace of tempered
%elliptic characters obtained from representations defined over $\OQ$.

%\item We write $\mcV_{G,\text{temp}} (ar) := \nu ( \mcK_{G,\text{temp}} (ar)) \subset
%\mcV_{G,\text{temp}}$.

\item Using only representations from $\mcR'_G$ in (16) gives $\mcK'_{G,\text{temp}} (ar)$.
% and $\mcV'_{G,\text{temp}} (ar)$.

\item J. Arthur also defined a notion of elliptic tempered virtual $G$-repre\-sen\-tations. 
These span $\mcK_{G,\text{temp}} (ar)$.
\end{enumerate}

\subsection{Results} \

The first general goal of this paper is a comparison of $\C G$-represen\-ta\-tions with
$\OQ G$-represen\-ta\-tions. It is known from \cite{Vig} that extension of scalars
from $\OQ$ to $\C$ preserves irreducibility of $\OQ G$-representations.
Of course not every irreducible $\C G$-representation comes from a 
$\OQ G$-representation. An obvious necessary condition is that its central cha\-rac\-ter
takes values in $\OQ^\times$.

But this condition does not suffice for a representation to be
defined over $\OQ$, since there are representations parabolically 
induced from representations $\bs$ of proper Levi subgroup $L$ such that the 
$Z(L)$-character of $\sigma$ does not take values in $\OQ^\times$. Looking for a 
description of a subspace of $\mcR_G$ coming from $\OQ$-representations we restrict 
our attention to  the part of the representation theory of $G$ that is "orthogonal" 
to the parabolically induced representations. To do that well, we pass from 
$\mcR_G$ to virtual $\C G$-representations, and more precisely to $\mcK_G$. Then the 
central character condition is accounted for by considering $\mcK'_G$ instead
of $\mcK_G$.

Arthur \cite{Ar} has shown that the subgroup of $\mcK_{G,\text{temp}}$ spanned by 
representations induced from proper parabolic subgroups of $G$ admits a complement 
$\mcK_{G,\text{temp}}(ar)$ of tempered elliptic representations which constitute 
the discrete, non-induced part of $\mcK_{G,\text{temp}}$.

The functor $\otimes_\OQ \C$ provides an embedding of $\mcK _{\OQ G}$ in 
$\mcK _G $. That enables 
us to define $\mcK_{\OQ G,\text{temp}} (ar) \subset \mcK_{\OQ G}$ as the $\OQ$-linear 
subspace of tempered elliptic representations obtained from $\C G$-representations
defined over $\OQ$. In Definitions (14)--(18) above the temperedness condition can be
omitted, which leads to natural "nontempered" analogues of these spaces. In 
parti\-cular there are spaces of elliptic (virtual) representations $\mcK_G (ar) 
\subset \mcK_G$ and $\mcK_{\OQ G} (ar) \subset \mcK_{\OQ G}$. See Paragraph 
\ref{par:7.2} for details. For all these spaces we can impose the additional 
condition that the central characters of the underlying representations must take 
values in $\OQ^\times$, which will be indicated by ${}'$.

\begin{theorem} \textup{(see Corollary \ref{cor:5.7})} 
\label{T2}
%Consider complex representations of $G$ in the category  $\mcR' _G$.
\begin{enumerate}
\item Every complex elliptic representation in $\mcR'_G$ can be realized over $\OQ$.
In particular every essentially square-integrable $\C G$-representation whose central
character takes values in $\OQ^\times$ can be realized over $\OQ$.
\item $\mcK'_{G,\text{temp}} (ar)$ equals $\mcK_{\OQ G,\text{temp}}(ar)$ and
$\mcK'_G (ar)$ equals $\mcK_{\OQ G} (ar)$.
%\item The space $\mcK'_G (ar) = \mcK_{\OQ G} (ar)$ is stable under $\Gal (\C / \mQ)$.
%\item $\nu (\mcK'_G (ar)) = \nu (\mcK_{\OQ G} (ar))$ is contained in $\OQ (G^{rs})$
%and stable under $\Gal (\C/ \mQ)$.
\end{enumerate}
\end{theorem}

For semisimple $p$-adic groups (or more generally reductive $p$-adic groups with 
compact centre), Theorem \ref{T2} says, in various ways, that the elliptic
part of $\mcK_G$ or $\mcK_{G,\text{temp}}$ is the same for representations over $\C$ and
over $\OQ$. Moreover this elliptic part is stable under Gal$(\C / \mQ)$.
For reductive $p$-adic groups, all this holds true because we restricted to $\mcR'_G$.\\

The second general goal of the paper is to show that well-known classes of 
$\C G$-representations are stable under the action of Gal$(\C / \mQ)$. 
It is clear from the definitions that the classes of
irreducible representations and of cuspidal representations 
are stable under Gal$(\C / \mQ)$. Clozel showed that the same 
holds for square-integrable representations (when char$(F)=0$), but his proof was
never worked out. We provide a different proof and we generalize his result to 
elliptic (virtual) $\C G$-representations, without conditions on char$(F)$.  

\begin{theorem} \label{thm:1.3}
\textup{(see Theorem \ref{Cl} and Corollary \ref{thm:5.7})}
\begin{enumerate}
\item The set of essentially square-integrable $\C G$-representations is stable
under $\Gal (\C / \mQ)$.
\item The set of elliptic $\C G$-representations is stable under $\Gal (\C / \mQ)$.
\item The spaces $\mcK_G (ar)$ and $\mcK'_G (ar) = \mcK_{\OQ G} (ar)$ are stable 
under $\Gal (\C / \mQ)$.
%\item The set of standard $\C G$-modules is stable under $\Gal(\C / \mQ)$.
\end{enumerate}    
\end{theorem}

We made the classes of representations in Theorem \ref{thm:1.3} into sets by considering
the objects up to the isomorphism. When char$(F) > 0$ and $G$ is semisimple, Theorem
\ref{thm:1.3}.(1) is claimed in \cite[Theorem 4.2]{CiHa}. The authors of \cite{CiHa}
sketch an argument, but their proof is still forthcoming.\\

Our main results have already been applied in two ways. In \cite{KaVa2} they have been
used to establish endoscopic decompositions of spaces of class functions on $G$, with
coefficients in various fields. In \cite{Sol4}, Theorem \ref{thm:1.3} has been generalized
to standard $\C G$-modules, which for instance enables one to define standard modules
with other coefficient fields. Both these applications fit nicely with approaches
to the local Langlands correspondence that involve representations over $\OQ_\ell$.\\

We conclude the introduction with a brief outline of the paper.
In Section~\ref{sec:4} the basic properties of (irreducible) $\OQ G$-representations 
are studied, and related to $\C G$.

Section \ref{sec:5} has an algebraic part, in which Theorem
\ref{T2}.(1) for essentially square-integrable representations is shown. In the analytic
part of Section \ref{sec:5}, we investigate formal degrees of square-integrable 
representations to prove that Gal$(\C / \mQ)$ permutes such representations.

Section \ref{sec:7} generalizes our results about essentially square-integrable
representations and finishes the proofs 
of Theorems \ref{T2} and \ref{thm:1.3}.

\section{Preliminaries about $G$-representations over $\OQ$ and over $\C$}\label{sec:4}

In this section we collect some results which we will use later, most of them come 
from \cite{Vig}.

The group of field automorphisms Gal$(\C / \mQ)$ 
act naturally on complex representations of $G$. Namely, for every  complex
representation $(\pi ,V)$ of $G$ and $\Gg \in \text{Gal}(\C / \mQ)$ we denote by 
$(\pi^{\Gg} ,V^{\Gg})$ the complex representation of $G$, where the vector space
$V^{\Gg}$ is defined as the base change $\C\otimes_{\Gg,\C} V$ 
with respect to the isomorphism $\Gg:\C\to\C$ and $G$ acts on $V^{\Gg}$ by the formula
$g(a\otimes v)= a\otimes gv$. 

This means that, if $\pi (g)_{ij}$ is the matrix of 
$\pi (g)$ with respect to some (possibly infinite) basis of $V$, then the matrix of 
$\pi^\gamma (g)$ with respect to the corresponding basis of $V^\gamma$ is 
$\gamma (\pi (g)_{ij})$.
This action of $\Gg$ respects (among others) characters, finite length and admissibility 
\cite[\S II.4.3]{Vig}.\\

We denote by $\mcH _G$ the Hecke algebra of compactly supported locally constant complex
valued measures on $G$ where the product is the convolution. A smooth complex representation
$(\pi ,V)$ of $G$ defines a representation of the algebra $\mcH _G$ on $V$, which
we also denote by $\pi$.

The category $\mcR _G$ is naturally equivalent to the full subcategory of the category of 
finite length nondegenerate $\mcH _G$-modules. We go freely from the language of complex 
$G$-representations to the language of $\mcH_G$-modules.

\begin{definition} Let  $K\subset G$ be an open compact subgroup.
\begin{enumerate}
\item $\mcH _{G,K} \subset \mcH _G $ is the subalgebra of $K$-bi-invariant measures.
\item For a representation $(\pi ,V)$ of $G$ we denote by $V^K\subset V$ the subspace of
$K$-invariant vectors, which is a $\mcH _{G,K}$-module.
\item $\mcH_{\OQ G}$ and $\mcH_{\OQ G,K}$ are the versions of $\mcH_{G}$ and $\mcH_{G,K}$
with coefficients in $\OQ$ instead of in $\C$.
\end{enumerate}
\end{definition}

\begin{proposition}\label{Kir} 
\textup{\cite[Proposition I.3.2]{Ren}} \\
Let $V$ be an irreducible $\C G$-module with $V^K \neq 0$. Then $V^K$ is
an irreducible $\mcH_{G,K}$-module and $V \cong \C [G/K] \otimes_{\mc H_{G,K}} V^K$. 

The same holds with coefficients in $\OQ$ instead of in $\C$.
\end{proposition}

Since $\OQ$ is algebraically closed and of characteristic zero, most (but not all) 
of the abstract representation theory of $G$ works the same over $\OQ$ as over $\C$. 
Schur's lemma does not apply automatically over $\OQ$, because the cardinality of 
$\OQ$ is not larger than that of $G/K$ for an open compact subgroup $K$ of $G$.
To prove it for $\OQ G$, one first has to establish admissibility of irreducible
representations, which is also not obvious.

\begin{theorem}\label{thm:3.4}
\textup{\cite[\S II.2.8]{Vig}} \\
Every irreducible representation of $\OQ G$ is admissible and has
a central character.
\end{theorem}

We will use Theorem \ref{thm:3.4} many times in this paper, sometimes implicitly.
Consider the functor
\begin{equation}
\otimes_\OQ \C : \Rep (\OQ G) \to \Rep (\C G).
\end{equation}

\begin{proposition}\label{prop:2.9}
Let $(\pi,V) \in \Irr (\OQ G)$.
\begin{enumerate}
\item The $\C G$-representation $(\pi, V \otimes_\OQ \C)$ is irreducible. 
\item Suppose that $(\pi',V') \in \Irr (\OQ G)$ and that $(\pi, V \otimes_\OQ \C)
\cong (\pi', V' \otimes_\OQ \C)$. Then $\pi \cong \pi'$.
\end{enumerate}
\end{proposition}
\begin{proof}
(1) See \cite[\S II.4.5]{Vig}. We take this opportunity to point that the proof in
\cite{Vig} is not entirely complete. Apart from the existence of a central character,
and also has to use that by \cite[\S II.2.4]{Vig} the representation 
$(\pi, V \otimes_\OQ \C)$ is completely reducible.\\
(2) Let $K \subset G$ be a compact open subgroup such that $V^K$ and $V'^K$ are nonzero. 
Let $e_K \in \mcH_{\OQ G}$ be the corresponding idempotent, then
\begin{equation}\label{eq:VK}
( V \otimes_\OQ \C)^K = \pi (e_K) (V \otimes_\OQ \C) =
(\pi (e_K) V) \otimes_\OQ \C = V^K \otimes_\OQ \C .
\end{equation}
By \eqref{eq:VK}, the $\mcH_{G,K}$-modules $V^K \otimes_\OQ \C$ and
$V'^K \otimes_\OQ \C$ are isomorphic. By the admissibility in Theorem \ref{thm:3.4} these 
are finite dimensional modules, so their 
traces are the same. For any $f \in \mcH_{\OQ G,K}$ we have
\[
\text{tr}(f,V^K) = \text{tr}(f,V^K \otimes_\OQ \C) = \text{tr}(f,V'^K \otimes_\OQ \C) =
\text{tr}(f,V'^K) .
\]
From Proposition \ref{Kir} we know that $V^K$ and $V'^K$ are irreducible. Since their
traces are the same, they are in fact isomorphic $\mcH_{\OQ G,K}$-modules. Now the
expression for $V$ in terms of $V^K$ in Proposition \ref{Kir} shows that 
$(\pi,V)$ and $(\pi',V')$ are isomorphic. 
\end{proof}

Proposition \ref{prop:2.9}.(1) can be generalized to finite length representations, in the
following sense.

\begin{lemma}\label{lem:3.12}
\textup{\cite[II.4.10.c]{Vig}} \\
Let $(\pi,V) \in \Rep (\C G)$ have finite length, and suppose that $\pi$ can be realized
over $\OQ$. Then every irreducible constituent of $\pi$ can be realized as an
irreducible $\OQ G$-representation.
\end{lemma}

Recall that a representation $(\pi ,V)$ of $\C G$ is cuspidal if, for every proper
pa\-ra\-bolic subgroup $P = L U_P$ of $G$, the space of $U_P$-coinvariants in $V$
is zero. Equivalently, $V$ is spanned by $\{ v - \pi (u) v : v \in V, u \in U_P \}$.
If $\Gg \in \Gal (\mC /\mQ)$ and $\pi$ is a cuspidal representation of $G$, 
then the re\-pre\-sentation $\pi ^\Gg$ is cuspidal \cite[\S II.4.3.c]{Vig}.

We note that $\otimes_\OQ \C$ sends cuspidal representations to cuspidal representations.
In fact, for cuspidal representations there is only little difference between working
over $\OQ$ and over $\C$:

\begin{proposition}\label{prop:3.13} 
\textup{\cite[\S II.4.9]{Vig}} \\
Let $(\pi, V) \in \Irr (\C G)$ be cuspidal. Then $\pi$ can be realized over $\OQ$
if and only if the central character cc$(\pi) : Z(G) \to \C^\times$ takes values 
in $\OQ^\times$.
\end{proposition}

Let $G^1 \subset G$ be the open normal subgroup generated by all compact subgroups of $G$.
Recall that a character of $G$ is called unramified if it is trivial on $G^1$.
We denote the group of unramified characters of $G$ by $X_{\nr}(G)$.
Proposition \ref{prop:3.13} relates to the following observation.

\begin{lemma} \label{lem:finim}
For every $\pi \in \Irr (\C G)$ there exists $\chi\in X_{\nr}(G)$ such that the 
central character cc$(\chi\otimes\pi)$ has finite image, thus takes values in 
$\{ z \in \OQ^\times : |z| = 1\}$.
\end{lemma}
\begin{proof}
Let $Z(G)^1$ be the unique maximal compact subgroup of $Z(G)$. Recall that the
group $G^1 Z(G) / G^1 \cong Z(G) / Z(G)^1$ is free abelian and
finitely generated. We fix a subset $\{x_1,\ldots,x_d\} \subset Z(G)$ whose image
in $Z(G)/Z(G)^1$ is $\mZ$-basis. We define
\begin{equation}\label{eq:3.20}
\mZ^d_G := \text{subgroup of } Z(G) \text{ generated by } \{x_1,\ldots,x_d\} .
\end{equation}
This is a group isomorphic to $\mZ^d$ and
\begin{equation}\label{eq:3.21}
G^1 Z(G) = G^1 \times \mZ^d_G .
\end{equation}
Since $Z(G)^1$ is compact and cc$(\pi)$ is smooth, the image cc$(\pi) (Z(G)^1)$ is a
finite subgroup of $\C^\times$. Define a character
$\chi' : Z(G) / Z(G)^1 \to \C^\times$ by $\chi' (x) =
\text{cc}(\pi)(x^{-1})$ for $x \in \mZ^d_G$. Then the map 
cc$(\pi) \otimes \chi' : Z(G) \to \C^\times$ sends  $\mZ^d_G$ to 1, so has finite image. 

The group $Z(G) / Z(G)^1$ embeds in $G / G^1$ with finite index. As $\C^\times$ is 
di\-vi\-sible, we can extend $\chi'$ to a $\C^\times$-valued character $\chi$ of $G / G^1$, or equivalently an unramified character of $G$. Then 
cc$( \pi \otimes \chi) = \text{cc}(\pi) \otimes \chi'$ has finite image.  
The last assertion follows from the observation that every finite subgroup of $\C^\times$ is contained in $\{ z \in \OQ^\times : |z| = 1\}$.
\end{proof}

For an arbitrary irreducible $\C G$-representation $\pi$ the cuspidal support Sc$(\pi)$
consists of a Levi subgroup $L \subset G$ and a cuspidal $\C L$-repre\-sen\-ta\-tion
$\sigma$, such that $\pi$ is a constituent of the parabolic induction of $\sigma$
(with respect to any parabolic subgroup with Levi factor $L$). This determines $(L,\sigma)$
uniquely up to $G$-conjugacy. We will also express that by saying that $(L,\sigma)$
represents Sc$(\pi)$.

The question whether or not $\pi$ can be
defined over $\OQ$ can be reduced to the cuspidal case:

\begin{proposition}\label{prop:3.14}
Let $(\pi,V) \in \Irr (\C G)$. Then $\pi$ can be realized over $\OQ$ if and only if its
cuspidal support can be realized over $\OQ$.
\end{proposition}

\begin{proof}
$\Rightarrow$ For a parabolic subgroup $P = L U$ of $G$, the (normalized) Jacquet
functor $J^G_P : \Rep (\C G) \to \Rep (\C L)$ is defined over $\OQ$. Hence
$J^G_P (\pi)$ can be realized over $\OQ$. It has finite length \cite[\S VI.6.4]{Ren},
so by Lemma \ref{lem:3.12} all its irreducible constituents can be realized over $\OQ$.
By construction \cite[\S VI.7.1]{Ren}, Sc$(\pi)$ can be represented by any of the
constituents of $J^G_P (\pi)$, for a suitable $P$. Hence Sc$(\pi)$ can be realized
over $\OQ$.\\
$\Leftarrow$ The (normalized) parabolic induction functor $I_P^G : \Rep (\C L) \to
\Rep (\C G)$ is defined over $\OQ$.
By definition, $\pi$ is an irreducible subquotient of $I_P^G (\sigma)$, where
$(L,\sigma)$ represents Sc$(\pi)$. By assumption $\sigma$ can be realized over $\OQ$, 
and hence $I_P^G (\sigma)$ can be realized over $\OQ$. It has a finite length 
\cite[Lemme VI.6.2]{Ren}, so by Lemma \ref{lem:3.12} all irreducible constituents of
$I_P^G (\sigma)$ can be realized over $\OQ$.
In particular $\pi$ can be realized over $\OQ$.
\end{proof}

\section{Essentially square-integrable representations}\label{sec:5}

\begin{definition}
\begin{enumerate}
\item A complex $G$-representation $(\pi, V)$ is unitary if
there exists a $G$-invariant po\-si\-tive definite Hermitian form $( , )$ on~$V$.
If $(\pi ,V)$ is irreducible then
such a form is unique up to multiplication by a positive scalar.
\item For a unitary irreducible representation $(\pi , V,(,))$ of $G$ and $v\in V$
we define a matrix coefficient $f_v(g):= (\pi (g)v,v)$. Then $f_v$ and $|f_v|$ are functions
on $G$, and $|f_v|$ can also be regarded as a function on $G/ Z(G)$.

\item An irreducible representation $ (\pi ,V)$ of $G$ is square-integrable (mo\-du\-lo
centre) if  it is unitary and $|f_v| \in L^2 ( G/Z(G) )$ for some non-zero $v \in V$.
Then $|f_v| \in L^2 ( G/Z(G))$ for all $v\in V$.

\item An irreducible representation $(\pi ,V)$ of $G$ is square-integrable (also called discrete
series) if $Z(G)$ is compact and $f_v \in L^2 (G)$ for all $v \in V$.
\item An irreducible representation $(\pi,V)$ is essentially square-integrable if
$\pi \otimes \chi$ is square-integrable (modulo centre) for some $\chi \in X_\nr (G)$.
\end{enumerate}
\end{definition}

\subsection{Some general results} \

Let $\mathcal G_\der$ be the derived group of $\mathcal G$ and write $G_\der =
\mathcal G_\der (F)$. Recall from \cite{Sil} that the restriction of an irreducible
$\C G$-representation to $G_\der$ is always a finite direct sum of irreducible representations.
The following lemma is well-known, we include a proof because we could not find a reference.

\begin{lemma}\label{lem:4.1}
Let $\pi$ be an irreducible $\C G$-representation. Then $\pi$ is essentially square-integrable
if and only if the restriction to $G_\der$ is a direct sum of square-integrable representations.
\end{lemma}
\begin{proof}
$\Rightarrow$ See \cite[Lemma 2.1 and Proposition 2.7]{Tad}. \\
$\Leftarrow$ The absolute value of the central character of $\pi$ extends uniquely to an 
unramified character of $G$ with values in $\RR_{>0}$, say $\nu_G$. Then 
$\pi \otimes \nu_G^{-1}$ has unitary central
character and its restriction to $G_\der$ is the same as that of $\pi$, so a
direct sum of square-integrable representations $\pi_1$. Every matrix coefficient $f$
of $\pi$ from a vector in $\pi_1$ is square-integrable on $G_\der$. Hence $|f|$ is
square-integrable on $G_\der Z(G) / Z(G)$. As $G / G_\der Z(G)$ is compact, it follows
that $|f|$ is also square-integrable on $G / Z(G)$. Hence $\pi \otimes \nu_G^{-1}$
is square-integrable modulo centre and $\pi$ is essentially square-integrable.
\end{proof}

For every parabolic subgroup $P = LU \subset G$, the normalized
parabolic induction functor $I_P^G$ gives a homomorphism $K(\mcR_L ) \to K(\mcR_G )$.
(It only depends on $L$, not on the choice of $P$ when $L$ is given.) Let
$K_{\mathrm{ind}}(\mcR_G)$ be the subgroup generated by the sets $I_P^G (K (\mcR_L))$,
where $L$ runs over all proper Levi subgroups of $G$.

\begin{lemma}\label{lem:4.6}
Let $\pi \in \Irr (\C G)$ be essentially square-integrable. Then the class of $\pi$ in
$K (\mcR_G)$ does not belong to $K_{\mathrm{ind}}(\mcR_G)$.
\end{lemma}
\begin{proof}
By assumption there exists a $\chi \in X_\nr (G)$ such that $\sigma := \pi \otimes \chi$ is
square-integrable modulo centre, and in particular tempered. The subgroup
$K_{\mathrm{ind}}(\mcR_G )$ is stable under tensoring by unramified characters of
$G$, so it suffices to prove that $\sigma$ does not lie in $K_{\mathrm{ind}}(\mcR_G )$.

Recall that an element $g \in G$ is called regular elliptic if $Z_G (g) / Z(G)$ is an
anisotropic torus. This is equivalent to requiring that $g$ does not belong to any
proper parabolic subgroup $P = LU$ of $G$.
For any finite length $\C L$-representation $\rho$, the Harish-Chandra character
$\chi_{I_P^G (\rho)}$ is supported on the set of elements of $G$ that are conjugate to
an element of $P$. Therefore $\chi_{I_P^G (\rho)}(g) = 0$ for any regular elliptic $g \in G$.
This also follows from \cite[Proposition 4.1]{KaVa}. Thus the character of any virtual
representation in $K_{\mathrm{ind}}(\mcR_G)$ vanishes at any regular elliptic element $g$.

According to \cite[Proposition 2.1.c]{Ar}, the square-integrable modulo centre representation
$\sigma$ is elliptic, in the sense that $\chi_\sigma (g) \neq 0$ for some regular elliptic
element $g \in G$. Therefore $\sigma$ does not belong to $K_{\mathrm{ind}}(\mcR_G)$.
\end{proof}

Recall that in our definitions square-integrability includes irreduciblity.
The set of square-integrable $\C G$-representations is countable (but empty when 
$Z(G)$ is not compact). To get from
there to finite sets of representations, one may involve compact open subgroups.

\begin{theorem}\label{thm:4.5}
Fix a compact open subgroup $K \subset G$.
\begin{enumerate}
\item The set of essentially square-integrable $\C G$-representations
$(\pi ,V)$ with $V^K \neq 0$ (considered up to isomorphism) forms a union
of finitely many $X_\nr (G)$-orbits.
\item There are only finitely many non-isomorphic square-integrable
$\C G^1$-representations $(\rho,W)$ with $W^K \neq 0$.
\end{enumerate}
\end{theorem}
\begin{proof}
(1) By \cite[Th\'eor\`eme VI.10.6]{Ren} there are only finitely many Bernstein
components in $\Irr (\C G)$ containing representations with nonzero $K$-fixed
vectors. Therefore we may restrict to essentially square-integrable
$G$-repre\-sen\-tations in a single Bernstein component.

In Lemma \ref{lem:4.6} we saw that no essentially square-integrable
$\C G$-repre\-sen\-tation $\pi$ lies in $K_{\mathrm{ind}}(\mcR_G)$. In the
terminology of \cite{BDK}, $\pi$ is discrete. According to
\cite[Proposition 3.1]{BDK}, the discrete irreducible $\C G$-repre\-sen\-ta\-tions
in one Bernstein component form a finite union of $X_\nr (G)$-orbits. That remains
true if we impose the additional condition of essential square-integrability.\\
(2) Let $(\rho,W)$ be a $\C G^1$-representation as in the statement.
Let $\mZ^d_G$ be as in \eqref{eq:3.20}.
By \eqref{eq:3.21} we can extend $(\rho,W)$ to a $G^1 Z(G)$-representation $\rho_Z$ 
by defining it to be trivial on $\mZ^d_G$. Since $\rho$ is square-integrable, 
$\rho_Z$ is square-integrable modulo centre.

As $[G : G^1 Z(G)]$ is finite, the $\C G$-representation $\mathrm{ind}_{G^1 Z(G)}^G
\rho_Z$ has finite length, and all its irreducible constituents
are square-integrable modulo centre. By Frobenius reciprocity $\rho$ is
a constituent of $\mathrm{Res}^G_{G^1} \omega$, for an irreducible constituent
$\omega$ of $\mathrm{ind}_{G^1 Z(G)}^G \rho_Z$ with nonzero $K$-fixed vectors.
From part (1) we know that there are only finitely many possibilities for
$\omega$, up to tensoring by unramified characters and up to isomorphism.
By \cite{Sil1}, $\mathrm{Res}^G_{G^1} \omega$ is a finite direct sum of
irreducible representations. As $\mathrm{Res}^G_{G^1} (\omega \otimes \chi) =
\mathrm{Res}^G_{G^1} \omega$ for $\chi \in X_\nr (G)$, there are only finitely
many possible $\rho$ (up to isomorphism).
\end{proof}

The next result generalizes Proposition \ref{prop:3.13}.

\begin{theorem}\label{thm:4.7}
Let $\pi \in \Irr (\C G)$ be essentially square-integrable.
\begin{enumerate}
\item $\pi$ can be realized over $\OQ$ if and only if cc$(\pi)$ takes values in $\OQ^\times$.
\item The $X_\nr (G)$-orbit of $\pi$ contains a $\pi' \in \Irr (\C G)$ such that the image of 
cc$(\pi')$ is finite subgroup of $\{ z \in \OQ^\times : |z| = 1\}$. Such a $\pi'$ can be
realized over $\OQ$ and is unitary and square-integrable (modulo centre).
\end{enumerate}
\end{theorem}
\begin{proof}
(1) The condition on the central character is clearly necessary.

Let $(\pi,V) \in \Irr (\C G)$ with $cc(\pi) : Z(G) \to \OQ^\times$.
By Proposition \ref{prop:3.13}, there exists a Levi subgroup $L\subset G$,
a cuspidal $\OQ L$-representation $\sigma$ and an unramified character $\chi \in
X_\nr (L)$ such that $\pi$ is an irreducible subquotient of $I_P^G (\sigma\otimes\chi)$.
By Proposition \ref{prop:3.14}, it suffices to show
that the character $\chi$ takes values in $\OQ^\times$.

For every $\gamma\in \mathrm{Gal}(\C / \OQ)$ the conjugate $\pi^{\gamma}$ is an
irreducible subquotient of $I_P^G (\sigma\otimes\chi^{\gamma})$. By Theorem
\ref{thm:4.5}.(1) all conjugates $\pi^\gamma$ lie in a finite number of
$X_\nr (G)$-orbits. With the uniqueness of the cuspidal support (up to
$G$-conjugation), it follows that all conjugates $\chi^{\gamma}$ of $\chi$
belong to a finite union of $X_\nr (G)$-orbits. Therefore the set of
conjugates  $\chi^\gamma |_{L\cap G^1}$ is finite, and thus the restriction
$\chi|_{L\cap G^1}$ has values in $\OQ^{\times}$.

By assumption, the central character $cc(\pi) = cc(\sigma)|_{Z(G)} \cdot
\chi|_{Z(G)}$ takes values in $\OQ^\times$. Since $\sigma$ is defined over $\OQ$,
the central character $cc(\sigma)|_{Z(G)}$ has values in $\OQ^\times$.

Thus, the restriction $\chi|_{Z(G)}$ and hence also $\chi|_{L\cap G^1 Z(G)}$
takes values in $\OQ^\times$. Since $[G:G^1 Z(G)] < \infty$, we conclude that
$\chi$ has values in $\OQ^\times$. \\
(2) By Lemma \ref{lem:finim}, there exists 
a $\chi\in X_\nr (G)$ such that the image of the central character of $\pi':= \pi \otimes \chi$ is a finite subgroup of $\{ z \in \OQ^\times : |z| = 1 \}$. By part (1), $\pi'$ can be
realized over $\OQ$. Further, since cc$(\pi')$ is unitary, we conclude that $\pi'$ is unitary and square-integrable (modulo centre).
\end{proof}

\subsection{Clozel's result about the Galois action}\label{par:Clozel} \

L. Clozel has shown  that the action of Gal($\C / \mQ)$ on $\Irr (\C G)$
preserves the discrete series. We thank him for sharing his unpublished proof 
with us. Clozel's argument uses adelic groups and automorphic representations,
and it is not clear to us how it can be made precise. Instead we worked out
a different proof, based on similar ideas but without adelic groups.

\begin{theorem}\label{Cl}
Let $(\pi ,V)$ be an essentially square-integrable $\C G$-repre\-sen\-tation and let
$\Gg \in \Gal (\mC /\mQ )$.
Then the representation $(\pi ^\Gg, V)$ is essentially square-integrable.
\end{theorem}

The proof of Theorem \ref{Cl} occupies the remainder of this paragraph.
Important tools to study square-integrable representations are formal degrees and the 
Plancherel measure, which exist for any locally compact unimodular group of type I.
Let $A_G$ be the maximal $F$-split torus in $Z(G)$.

\begin{theorem}\label{thm:formalDegree} 
\textup{\cite[Th\'eor\`eme 14.3.3]{Dix} and \cite[p. 265]{Wal}}  \\
Let $(\pi ,V)$ be an irreducible square-integrable modulo centre $G$-representation.
Then there exists a unique Haar measure $\mu_\pi$ on $G / A_G$ such that
\[
\int_{g\in G / A_G} |f_v (g)|^2 \textup{d}\mu_\pi (g) = |v|^4 \quad
\text{for any } v\in V .
\]
\end{theorem}

We fix a Haar measure $\mu_\mQ$ on $G / A_G$ such that $\mu_\mQ (K' ) \in 
\mQ_{>0}$ for every compact open subgroup $K'$ of $G / A_G$. We define the formal 
degree of an irreducible square-integrable representation $\pi$ as 
\[
\fdeg (\pi) = \frac{\textup{d}\mu_\pi}{\textup{d}\mu_\mQ}.
\]
Notice that $\fdeg (\pi)$ is invariant under twisting $\pi$ by unitary unramified 
characters of $G$. That allows us to define the formal degree of an essentially
square-integrable $\C G$-representation $\pi'$ as
\[
\fdeg (\pi') = \fdeg (\pi' \otimes \chi) \text{ for any } \chi \in X_\nr (G) 
\text{ such that } \pi' \otimes \chi \text{ is unitary.}
\]
If an irreducible representation of $G$ is not essentially square-integrable, 
then we define its formal degree to be zero.
The following result will be useful in the next section.

\begin{theorem}\label{thm:4.8} \textup{\cite{Mey}} \\
The formal degree of any essentially square-integrable $G$-representation 
$\pi$ lies in $\mQ_{>0}$. 
\end{theorem}
\begin{proof}
We may assume that $\pi$ has a unitary central character, so that it is square-integrable 
modulo centre. According to \cite[Lemma 36 and Theorem 38]{Mey}, $\fdeg (\pi) \in \mQ_{>0}$.
However, in \cite{Mey} another definition of formal degrees is used, which is equivalent to
the definition from Theorem \ref{thm:formalDegree} with $A_G$ replaced by $Z^\circ (G)$ (where 
$\circ$ means the connected component as algebraic group). But, as $A_G$ is cocompact in 
$Z^\circ (G)$, all the results from \cite{Mey} also work with $A_G$ instead of $Z^\circ (G)$. 
\end{proof}

When $G$ is semisimple or more generally when $Z(G)$ is compact, formal degrees 
also arise from the Plancherel measure. This measure $\mu_{Pl}$ is usually defined 
via unitary $G$-representations which are not smooth.
We recall from \cite[(95)]{SolQ} that there is a canonical bijection between
the set of irreducible unitary $G$-representations on Hilbert spaces and the set
of irreducible smooth unitary $G$-representations (both considered up to isomorphism).
In one direction the functor is completion, in the opposite direction the functor is
taking smooth vectors. Therefore we may replace unitary $G$-representations by
smooth unitary $G$-representations for $\mu_{Pl}$. Further, we may
extend the Plancherel measure to $\Irr (\C G)$ by defining it to be supported 
on the unitary representations in $\Irr (\C G)$. That leads to:

\begin{theorem}\textup{\cite[\S 18.8]{Dix}} 
\label{thm:Plancherel} \\
Assume that $Z(G)$ is compact. There exists a unique measure $\mu_{Pl}$ on 
$\Irr (\C G)$ which is supported on unitary representations and satisfies
\[
f(1) = \int_{\Irr (\C G)} \mathrm{tr} \, \pi (f \mu_\mQ ) \, 
\textup{d} \mu_{Pl} (\pi) \qquad \text{for all } f \mu_\mQ \in \mcH_G . 
\]
An irreducible $\C G$-representation $\pi$ is square-integrable if and only if\\ 
$\mu_{Pl}(\{\pi\}) > 0$. Moreover $\fdeg (\pi) = \mu_{Pl}(\{\pi\})$.
\end{theorem}

We focus first on the cases of Theorem \ref{Cl} where $G$ is semisimple and the underlying
field $F$ is $p$-adic (so of characteristic zero). By \cite[Remark 1.13]{BoHa} there
exists a number field $\mE$ such that all infinite places of $\mE$ are real and
$\mE_v \cong F$ for a finite place $v$. Let $\mcV_\mE$ be the set of places
of $\mE$ and let $\mcV_\mE^\infty$ be the subset of infinite places. Recall that
any semisimple real group has a compact form. By \cite[Theorem B]{BoHa}, there exists
an $\mE$-group $\mcG$ such that $\mcG (\mE_v) \cong \mc G (F) \cong G$ and such that
$\mcG (\mE_w)$ is compact for every $w \in \mcV_\mE^\infty$.

\begin{theorem}\label{thm:4.2} 
\textup{\cite[Theorem A and page 60]{BoHa}} \\
Let $\Gamma \subset \mcG (\mE)$ be an $S$-arithmetic subgroup, where 
$S = \mcV_\mE^\infty \cup \{v\}$. Then $\Gamma$ is a discrete cocompact subgroup
of the semisimple group $\mcG (\mE_v)$ over the $p$-adic field $\mE_v \cong F$.
\end{theorem}

By varying $\Gamma$ in Theorem \ref{thm:4.2}, we can construct a sequence
$(\Gamma_n )_{n=0}^\infty$ such that:
\begin{itemize}
\item each $\Gamma_n$ is a discrete cocompact subgroup of $\mcG (\mE_v) \cong G$,
\item $\Gamma_{n+1} \subset \Gamma_n$ and $\bigcap_n \Gamma_n = \{1\}$,
\item each $\Gamma_n$ is normal in $\Gamma_0$.
\end{itemize}
For example if $\mcG$ is defined over the ring of integers $\mathfrak o_\mE$, we can take
\[
\Gamma_0 = \mcG (\mathfrak o_{\mE,S}) \;\subset\; \bigcap\nolimits_{w \in S} \mcG (\mE_w) 
\cap \bigcap\nolimits_{w \in \mcV_\mE \setminus S} \mcG (\mathfrak o_{\mE_w}) ,
\]
and define $\Gamma_n$ by putting more congruence conditions on the places in
$\mcV_\mE \setminus S$. For each $n \in \ZZ_{\geq 0}$ we have a unitary 
representation of $G$ on $L^2 (\Gamma_n \backslash G)$, by right translations.
The following theorem is a reformulation of a result of Sauvageot \cite{Sau}.

\begin{theorem}\label{thm:4.3}
Take $\Gamma_n$ as above.
\begin{enumerate}
\item For any irreducible unitary $G$-representation $\tilde \pi$ on a
Hilbert space:
\[
\mu_{Pl}(\{ \tilde \pi \}) = \lim_{n \to \infty} \mathrm{vol}(\Gamma_n 
\backslash G)^{-1} \dim \Hom_{\C G} (\tilde \pi, L^2 (\Gamma_n \backslash G)) .
\]
\item For any $\pi \in \Irr (\C G)$:
\[
\mu_{Pl}(\{ \pi \}) = \lim_{n \to \infty} \mathrm{vol}(\Gamma_n \backslash G)^{-1}
\dim \Hom_{\C G} (\pi, C^\infty (\Gamma_n \backslash G)) .
\]
\end{enumerate}
\end{theorem}
\begin{proof}
(1) See \cite[Introduction]{Sau}.\\
(2) It follows from the remarks before Theorem \ref{thm:Plancherel} that in 
part (1) we may replace both $\tilde \pi$ and 
$L^2 (\Gamma_n \backslash G)$ by their smooth parts. That does not change the 
dimension of the Hom-space because $\tilde \pi$ is irreducible. The
smooth part of $L^2 (\Gamma_n \backslash G)$ is $C^\infty (\Gamma_n \backslash G)$
and the smooth part of $\tilde \pi$ is an irreducible unitary 
representation in $\Rep (\C G)$ \cite[(95)]{SolQ}. Moreover any irreducible
unitary smooth $G$-representation can be obtained in this way, which proves 
(2) when $\pi$ is unitary. 

When $\pi$ is not unitary, $\mu_{Pl}(\{\pi\}) = 0$ by definition. 
$\Hom_{\C G}(\pi, C^\infty (\Gamma_n \backslash G))$ is also zero because all
subrepresentations of $C^\infty (\Gamma_n \backslash G)$ are unitary.
\end{proof}

The next result is the crucial part of Clozel's argument.

\begin{proposition}\label{prop:Clozel}
Let $\pi$ be a discrete series representation of a semisimple group $G$ over 
a $p$-adic field, and let $\Gg \in \Gal (\mC /\mQ )$. Then the representation 
$\pi ^\Gg$ is also a discrete series and has the same formal degree as~$\pi$.
\end{proposition}
\begin{proof}
Note that the $\C$-vector space $C^\infty (\Gamma_n \backslash G)$ has a natural $\QQ$-structure, and 
there is an isomorphism of $\C G$-representations
\begin{equation}
\begin{array}{ccc}
C^\infty (\Gamma_n \backslash G) & \to & C^\infty (\Gamma_n \backslash G)^\gamma \\
f & \mapsto & \gamma^{-1} \circ f 
\end{array}.
\end{equation}
Hence the representation of $G$ on $C^\infty (\Gamma_n \backslash G)$ is stable under the
action of $\Gal (\C / \mQ)$. By Theorem \ref{thm:4.3}.(2) 
\begin{align*}
\mu_{Pl}(\{ \pi^\gamma \}) & = \lim_{n \to \infty} \mathrm{vol}(\Gamma_n \backslash G)^{-1}
\dim \Hom_{\C G} (\pi^\gamma , C^\infty (\Gamma_n \backslash G)) \\
& = \lim_{n \to \infty} \mathrm{vol}(\Gamma_n \backslash G)^{-1}
\dim \Hom_{\C G} (\pi , C^\infty (\Gamma_n \backslash G)) = \mu_{Pl}(\{ \pi \}) .
\end{align*}
Now Theorem \ref{thm:Plancherel} and the square-integrability of $\pi$ imply that
$\pi^\gamma$ is square-integrable.
\end{proof}

Unfortunately, Theorem \ref{thm:4.3} cannot be applied to semisimple groups $G'$ over 
local function fields, because it is not known whether such $G'$ have discrete
cocompact subgroups. To establish Theorem \ref{Cl} for such groups, 
we use the method of close local fields from \cite{Del,Gan,Kaz}. Let $F'$ be a local field
of positive characteristic and write $G' = \mcG' (F')$. Let $G'_{x',0}$ be the parahoric
subgroup of $G'$ associated to a special vertex $x'$ of the Bruhat--Tits building of $G'$.
For $m \in \ZZ_{> 0}$, let $G'_{x',m}$ be the Moy--Prasad subgroup of level $m$. We will
choose $m$ so large that the representations which we consider have nonzero 
$G'_{x',m}$-fixed vectors. 

Let $\Rep ( \C G', G'_{x',m})$ be the category of $\C G'$-modules which are generated by
their $G'_{x',m}$-fixed vectors. Recall from \cite[Proposition VI.9.4]{Ren} that there
is an equivalence of categories
\begin{equation}\label{eq:4.1}
\begin{array}{ccc}
\Rep (\C G', G'_{x',m}) & \to & \Rep (\mcH_{G',G'_{x',m}}) \\
V & \mapsto & V^{G'_{x',m}}
\end{array}.
\end{equation}
We recall that a non-archimedean local field $F$ is said to be $l$-close to $F'$ if 
there is a ring isomorphism $\mathfrak o_{F'} / \mathfrak p_{F'}^l \cong 
\mathfrak o_F / \mathfrak p_F^l$. Given $F'$ and $l \in \ZZ_{\geq 0}$, there always 
exists an $l$-close $p$-adic field $F$ \cite{Del}.  

\begin{theorem}\label{thm:4.4}
\textup{\cite[Theorem 4.1]{Gan}} \\
Suppose that $F'$ and $F$ are $l$-close, where $l$ is large enough compared to $m$.
Then there exist
\begin{itemize}
\item a semisimple $F$-group $\mcG$,
\item a special vertex $x$ of the Bruhat--Tits building of $G = \mcG (F)$,
\item a bijection $G'_{x',m} \backslash G' / G'_{x',m} \to G_{x,m} \backslash G / G_{x,m}$,
\end{itemize}
which induce an $\C$-algebra isomorphism $\mcH_{G',G'_{x'm}} \cong \mcH_{G,G_{x,m}}$.
\end{theorem}

Theorem \ref{thm:4.4} also holds for reductive groups, but we do not need that here.
From \eqref{eq:4.1} (for $G'$ and for $G$) and Theorem \ref{thm:4.4}, we obtain an
equivalence of categories
\begin{equation}\label{eq:4.2}
\Rep (\C G', G'_{x',m}) \cong \Rep (\C G, G_{x,m}) .    
\end{equation}

\begin{lemma}\label{lem:4.9}
The equivalence of categories \eqref{eq:4.2}:
\begin{enumerate}
\item commutes with the action of $\Gal (\C / \mQ)$,
\item preserves square-integrability of irreducible representations.
\end{enumerate}
\end{lemma}
\begin{proof}
(1) This holds because the functors in \eqref{eq:4.1} and the algebra isomorphism in
Theorem \ref{thm:4.4} commute with the action of $\Gal (\C / \mQ)$.\\
(2) For general linear groups over division algebras, this is 
\cite[Th\'eor\`eme 2.17.b]{Bad}. That proof can be adapted to the setting of \cite{Gan},
but we prefer a different argument. We regard $\mcH_{G',G'_{x',m}}$ as the algebra of
compactly supported functions on $G'_{x',m} \backslash G' / G'_{x',m}$, where
$G'_{x',m}$ has volume 1. We make it into a unital Hilbert 
algebra with *-operation $f^* (g) = \overline{f(g^{-1})}$ and trace $f \mapsto f(1)$. 
Then it has a $C^*$-completion and a Plancherel measure \cite[\S 3.2]{BHK}. Now
Theorem \ref{thm:4.4} provides an isomorphism of Hilbert algebras 
$\mcH_{G',G'_{x'm}} \cong \mcH_{G,G_{x,m}}$, so it preserves the Plancherel measures.
Furthermore, by \cite[Theorem 3.5]{BHK} the functor \eqref{eq:4.1} preserves 
Plancherel measures, up to positive real factor that depends on the normalization
of the Haar measure on $G'$. The same holds with $G$ instead of $G'$. It follows that
the equivalence of categories \eqref{eq:4.2} preserves Plancherel measures, up to
a positive real factor. That and Theorem \ref{thm:Plancherel} entail that \eqref{eq:4.2}
preserves square-integrability of irreducible representations.
\end{proof}

We are ready to finish the proof of Theorem \ref{Cl}, first for semisimple groups
and then in general.\\

\emph{Proof of Theorem \ref{Cl}}\\
Let $G'$ be as in Theorem \ref{thm:4.4} and let $\pi' \in \Irr (\C G')$ be 
square-integrable. We choose $m \in \ZZ_{>0}$ such that $\pi'$ has nonzero 
$G'_{x',m}$-fixed vectors. Let $G$ be the semisimple group from Theorem 
\ref{thm:4.4}, over a $p$-adic field $F$, and
let $\pi \in \Irr (\C G)$ be the image of $\pi'$ under \eqref{eq:4.2}. By Lemma
\ref{lem:4.9}.(2), $\pi$ is square-integrable. For $\gamma \in \Gal (\C / \mQ)$,
Lemma \ref{lem:4.9}.(1) says that the image of $\pi'^{\gamma}$ under \eqref{eq:4.2}
is $\pi^\gamma$. By Proposition \ref{prop:Clozel}, $\pi^\gamma$ is square-integrable.
Now Lemma \ref{lem:4.9}.(2) says that $\pi'^{\gamma}$ is square-integrable.

This and Proposition \ref{prop:Clozel} establish Theorem \ref{Cl} for all 
semisimple groups over non-archimedean local fields.

Consider any reductive group $G$ over a non-archimedean local field, and an essentially
square-integrable $\pi \in \Irr (\C G)$. By Lemma \ref{lem:4.1} the restriction of $\pi$ 
to the derived group $G_\der$ is a direct sum of irreducible representations $\pi_1$, 
each of which is discrete series. By Theorem \ref{Cl} for semisimple groups, 
$\pi_1^\gamma$ is square-integrable. Thus $\pi^\gamma$ is
an irreducible $G$-representation whose restriction to $G_\der$ is a direct sum of
square-integrable representations $\pi_1^\gamma$.  By Lemma \ref{lem:4.1},
$\pi^\gamma$ is essentially square-integrable. $\quad \Box$

\section{Elliptic representations}\label{sec:7}

The goal of this section is to generalize the results from Section \ref{sec:5}
to %elliptic characters 
elliptic representations. Before we come to their definition, 
we recall the notion of $R$-groups from \cite{Ar} that will be essential in this section.

\subsection{Intertwining operators and $\mu$-functions}\ 
\label{sec:6}

First we recall the definition of Harish-Chandra's intertwining operators. Consider two parabolic
subgroups $P = L U_P$ and $P' = L U_{P'}$ with a common Levi factor $L$. Let 
$(\sigma,V_\sigma)$ be an irreducible  $\C L$-representation. All the representations $I_P^G (\sigma \otimes \chi)$
with $\chi \in X_\nr (L)$ can be realized on the same vector space, namely
$\mathrm{ind}_{P \cap K_0}^{K_0} V_\sigma$ for a good maximal compact subgroup $K_0$ of $G$. This
makes it possible to speak of objects on $I_P^G (\sigma \otimes \chi)$ that vary regularly or
rationally as functions of $\chi \in X_\nr (L)$. Consider the intertwining operators
\begin{equation}\label{eq:6.1}
\begin{array}{llll}
J_{P'|P}(\sigma \otimes \chi) : & I_P^G (\sigma \otimes \chi) & \to & I_{P'}^G (\sigma \otimes \chi) \\
 & f & \mapsto & [g \mapsto \int_{U_P \cap U_{P'} \setminus U_{P'}} f(ug) \, \textup{d} u ]
\end{array}.
\end{equation}
This is well-defined as a family of $G$-homomorphisms depending
rationally on $\chi \in X_\nr (L)$ \cite[Th\'eor\`eme IV.1.1]{Wal}. There is an
alternative construction of \eqref{eq:6.1}, in \cite[proof of Th\'eor\`eme IV.1.1]{Wal}.

Assume now that $L \subset G$ is a maximal proper Levi subgroup, and  
let $\bar P = L U_{\bar P}$ be the parabolic subgroup opposite to $P = L U_P$. 
We consider the composition 
\begin{equation}\label{eq:6.9}
J_{P | \bar P}(\sigma \otimes \chi) J_{\bar P | P}(\sigma \otimes \chi) :
I_P^G (\sigma \otimes \chi) \to I_P^G (\sigma \otimes \chi) \qquad \chi \in X_\nr (L) . 
\end{equation}
This depends rationally on $\chi$, and for generic $\chi$ the representation 
$I_P^G (\sigma \otimes \chi)$ is irreducible \cite[Th\'eor\`eme 3.2]{Sau}. Therefore
\eqref{eq:6.9} is a scalar operator \cite[\S IV.3]{Wal}, say 
\begin{equation}\label{eq:6.36}
j_{G,L}(\sigma \otimes \chi) \mathrm{id} \quad \text{with} \quad 
j_{G,L} : X_\nr (L) \sigma \to \C \cup \{\infty\} .
\end{equation}
Let now $P = L U_P$ be an arbitrary parabolic subgroup of $G$. Let $\Phi (G,Z(L))$ 
be the set of reduced roots of $(G,Z(L))$ and let $\Phi (G,Z(L))^+$ be the subset of 
roots appearing in the Lie algebra of $P$. 
For $\alpha \in \Phi (G,Z(L))^+$, let 
$U_\alpha$ (resp. $U_{-\alpha}$) be the root subgroup of $G$ for all positive
(resp. negative) multiples of $\alpha$. Let $L_\alpha$ be the Levi subgroup of
$G$ generated by $L \cup U_\alpha \cup U_{-\alpha}$. Then $L$ is a maximal proper
Levi subgroup of $L_\alpha$, so the earlier results from this section apply to
$L \subset L_\alpha$. By definition \cite[\S V.2]{Wal} Harish-Chandra's $\mu$-function
in this corank one setting is
\begin{equation}\label{eq:6.15}
\mu_{L_\alpha,L}(\sigma \otimes \chi) = \gamma ( L_\alpha | L)^2 
j_{L_\alpha,L}(\sigma \otimes \chi)^{-1},
\end{equation}
where the constant $\gamma (L_\alpha | L) \in \mQ_{>0}$ is defined in \cite[p. 241]{Wal}.
%We rewrite $\chi_-$ from \eqref{eq:6.35} for $L_\alpha \supset L$ as $\chi_-^\alpha$
%and we define
%\[
%(\chi_-^\alpha )^* \mu_{L_\alpha,L}(\sigma \otimes \chi) = 
%\mu_{L_\alpha,L}(\sigma \otimes \chi_-^\alpha \chi).
%\]

%\begin{lemma}\label{lem:6.12}
%\begin{enumerate}
%\item In the setting of Theorem \ref{thm:6.2}.(2), $\mu_{L_\alpha,L}$ is 
%a rational function defined over $\Qq$ and $\mu_{L_\alpha,L} + (\chi_-^\alpha )^* 
%\mu_{L_\alpha,L}$ is a rational function defined over $\mQ$.
%If moreover $\delta_{L U_\alpha}$ takes values in $\mQ^{\times 2}$ or that $\chi_-^\alpha 
%(h_\alpha^\vee) = 1$, then $\mu_{L_\alpha,L}$ is a rational function defined over $\mQ$.
%\item In the setting of Theorem \ref{thm:6.2}.(1), $\mu_{L_\alpha,L}$ is constant on 
%$X_\nr (L) \sigma$, with value in $\RR_{>0}$.
%\end{enumerate}
%\end{lemma}
%\begin{proof}
%(1) The first assertion follows from Theorem \ref{thm:6.4} and Lemma \ref{lem:6.11}.
%Either of the two additional assumptions implies that $(\chi_-^\alpha)^* \mu_{L_\alpha,L} = 
%\mu_{L_\alpha,L}$. \\
%(2) This follows directly from Theorem \ref{thm:6.2} and \eqref{eq:6.15}.
%\end{proof}

%Let $P = L U_P$ be a parabolic subgroup of $G$. 

\subsection{R-groups and elliptic representations} \
\label{par:7.2}

This paragraph is based on constructions and results of J. Arthur presented in \cite{Ar}.
We note that Arthur considered only reductive groups defined over local fields of
characteristic zero. But all the harmonic analysis which is necessary for these results
has also been developed for reductive groups over local fields of positive characteristic,
so we may include those as well.

As before, let $L\subset G$ be a Levi subgroup, and let the group $N_G (L)$ acts on 
$\Irr (\C L)$, by $(n \cdot \sigma)(l) = \sigma (n^{-1} l n)$. This action descends to an action of 
\[
W_L := N_G (L) / L
\] 
on $\Irr (\C L)$,  % and on $\Irr (\OQ L)$, which sends $X_\nr (L)$ to $X_\nr (L)$.
and we denote by $W_{L,\delta}\subset W_L$ the stabilizer of $\delta \in \Irr (\C L)$. 

Let $X_*(Z(L))$ and $X_*(Z(G))$ be the groups of cocharacters (over $F$). %To define elliptic representations, we will also make use of elliptic elements in Weyl groups.
We say that 
\begin{equation}\label{eq:5.4}
w \in W_L \quad \text{is elliptic if} \quad
X_* (Z(L))^w = X_* (Z(G)).
\end{equation}
We let $W_{L,ell}$ be the set of elliptic elements of $W_L$ and for every $\delta\in \Irr(\C L)$ we write
\[
W_{L,\delta,ell}:= W_{L,\delta} \cap W_{L,ell}.
\]

Let $\delta \in \Irr (\C L)$ be a square-integrable representation (modulo centre). %Since we will consider the 
%entire $X_\nr (L)$-orbit of $\delta$, we may assume by Theorem \ref{thm:4.7} that $\delta$ 
%can be realized over $\OQ$ and that it is unitary and square-integrable modulo centre.
Consider the set of reduced roots $\alpha$ of $(G,Z(L))$ such that Harish-Chandra's function
$\mu_{L_\alpha,L}$ (see \cite[\S V.2]{Wal} or the previous paragraph) has a zero at $\delta$. These roots form a finite integral root system, say $\Phi_\delta$. The group $W_{L, \delta}$ acts on $\Phi_\delta$ and contains the Weyl group
$W(\Phi_\delta )$ as a normal subgroup. Let $\Phi_\delta^+$ be the
positive system of roots appearing in the Lie algebra of $P$. The R-group
$R_\delta$ is defined as the stabilizer of $\Phi_\delta^+$ in
$W_{L,\delta}$. Since $W(\Phi_\delta )$ acts simply transitively
on the collection of positive systems in $\Phi_\delta$, we have a decomposition
\begin{equation}\label{eq:5.5}
W_{L,\delta} = W(\Phi_\delta ) \rtimes R_\delta .
\end{equation}

Let $P=LU_P$ be a parabolic subgroup of $G$ and consider the $G$-represen\-ta\-tion
$I_P^G (\delta)$ obtained by normalized parabolic induction. As in \cite{Ar}, there 
exists a canonical projective representation 
$J_\delta  : R_\delta \to \text{Aut}_{G} ( I_P^G (\delta))/\C^{\times}$. 

By the theory of projective representations, $J_{\delta}$ gives rise to an element in the cohomology group  
$\eta_{\delta}\in H^2(R_{\delta},\C^{\times})$, and there exists a central extension 
\begin{equation} \label{extension}
1\to Z_\delta \to \ti R_\delta \overset{\tau}{\to} R_\delta \to 1
\end{equation}
of finite groups such that the image $\ti\eta_{\delta}\in H^2(\ti R_{\delta},\C^{\times})$ of $\eta_{\delta}$ vanishes. 

Hence, 
the projective representation $J_\delta\circ \tau: \ti R_\delta \to Aut_{G} ( I_P^G (\delta))/\C^{\times}$ 
has a lift to an (actual) representation  $\tilde J_\delta  : \ti R_\delta \to Aut_{G} ( I_P^G (\delta))$, 
unique up to a character $\ti R_\delta\to \C^{\times}$. In particular, we have $\tilde J_\delta |_{Z_\delta} = \zeta \,$Id 
for some character $\zeta : Z_\delta \to \mC ^\times$.

Let $\Pi (\delta, \zeta)$ be the set of irreducible representations of the group $\ti R_\delta$ whose restriction on $Z_\delta$ is equal to $\zeta \,$Id. Then there is an isomorphism of $\ti R_\delta \times G$-representations
\begin{equation}\label{eq:5.6}
I_P^G (\delta) \cong \bigoplus\nolimits_{\kappa \in \Pi (\delta,\zeta)} \kappa \otimes I_P^G (\delta)_\kappa ,
\end{equation}
where each $I_P^G (\delta)_\kappa$ is irreducible. 

%This is a more precise version of \eqref{eq:7.9}, which facilitates working with $\OQ$-coefficients. In \eqref{eq:5.6}, 
%the central idempotent $e_\zeta \in \C [Z_\delta] \subset
%\C [\ti R_\delta]$ associated to $\zeta \in \Irr (Z_\delta)$ gives rise to an algebra 
%isomorphism $e_\zeta \C [\ti R_\delta] \cong \C [R_\delta, \natural_\delta]$.
%The representation $\ti J_\delta$ lifts $J_\delta$ from \eqref{eq:7.9}, reconciling the
%differences between the notations in \eqref{eq:5.6} and \eqref{eq:7.9}.

For $r\in R_{\delta}$ we choose a preimage $\ti r\in \tau^{-1}(r)\subset\ti R_{\delta}$, and consider the virtual $G$-representation
\begin{equation}\label{eq:5.7}
I_P^G (\delta)_{r} = I_P^G (\delta)_{(\tau,\ti J_{\delta},\ti r)}:=\sum\nolimits_{ \kk \in \Pi (\delta, \zeta) } \text{tr}
(\kk (\ti r)) I_P^G (\delta)_\kk \; \in\mcK_G. 
%\quad \text{given by} \quad g \mapsto \tilde J_\delta (\ti r) I_P^G (\delta) (g).
\end{equation}

Notice that the group  $\ti R_\delta$ is finite. Therefore for every $\tilde r\in \ti R_\delta$ and  $\kk \in \Pi (\delta, \zeta)$ we have $\text{tr} (\kk (\ti r)) \in \OQ$, thus $I_P^G (\delta)_{r}$ indeed belongs to $\mcK_G$. Notice that a priori $I_P^G (\delta)_{r}$ depends on a triple 
$(\tau,\ti J_{\delta},\ti r)$. 

\begin{lemma} \label{lem:indep}
The virtual representation $I_P^G (\delta)_{r}$ is defined uniquely up to a scalar in $\OQ^{\times}$ (a root of unity). 
\end{lemma}

\begin{proof}
To show the independence of $\ti r$, we fix a pair $(\tau, \ti J_{\delta})$. Then for every $\ti r',\ti r\in \tau^{-1}(r)$ 
we have $\ti r'=z\ti r$ for some $z\in Z_{\delta}$. Then 
$\ti J_{\delta}(z)=\zeta(z)$Id, so the independence of $\ti r$ follows 
from the equality 
$I_P^G (\delta)_{(\tau,\ti J_{\delta},\ti r')}=\zeta(z)\cdot I_P^G (\delta)_{(\tau,\ti J_{\delta},\ti r)}$. 

To show the independence of $\ti J_{\delta}$, we fix $\tau$. Then for every two lifts\\ $\ti J'_{\delta},\ti
J_{\delta}:\ti R_{\delta}\to \Aut_G( I_P^G (\delta))$ of $\tau\circ J_{\delta}$
there exists a character $\chi:\ti J_{\delta}\to \C^{\times}$ such that 
$\ti J'_{\delta}(\ti r)=\ti J_{\delta}(\ti r)\cdot \chi(\ti r)$ for every 
$\ti r\in \ti R_{\delta}$. Then the independence of $\ti J_{\delta}$ 
follows from the equality $I_P^G (\delta)_{(\tau,\ti J'_{\delta},\ti r)}=\chi(\ti r)
\cdot I_P^G (\delta)_{(\tau,\ti J_{\delta},\ti r)}$. 

Next, assume that $\tau:\ti R_{\delta}\to R_{\delta}$ and $\tau':\ti R'_{\delta}\to R_{\delta}$ 
are two central extensions such that $\tau'$ dominates $\tau$, that is, there exists homomorphism 
$\nu:\ti R'_{\delta}\to \ti R_{\delta}$ such that $\tau'=\tau\circ\nu$. Then for every 
lift $\ti J_{\delta}$ of $\tau\circ J_{\delta}$, the composition $\ti J'_{\delta}:=\ti J_{\delta}\circ \nu$ is a lift of $\tau'\circ J_{\delta}$. Then for every $\ti r'\in \ti R'_{\delta}$ we have an equality $I_P^G (\delta)_{(\tau',\ti J'_{\delta},\ti r')}= 
I_P^G (\delta)_{(\tau,\ti J_{\delta},\nu(\ti r'))}$, and the independence of the central extension  follows in this case. 

Finally, to show the independence of the central extension in general, 
we notice that for any two central extensions
 $\tau:\ti R_{\delta}\to  R_{\delta}$ and $\tau':\ti R'_{\delta}\to
R_{\delta}$, there is a third one $\tau'':\ti R_{\delta}\times_{R_{\delta}}\ti R'_{\delta}\to  R_{\delta}$, which dominates the first two. 
\end{proof}

%\texttt{This is false!}

%Furthermore, (up to a scalar in $\OQ^{\times}$), 
%$I_P^G (\delta)_{r}$ is independent of a choice of a central extension (\ref{extension}). 

%\texttt{In fact $I_P^G (\delta)_{r}$ really depends on the lift $\tilde J_\delta$,
%which depends on $\tilde R_\delta$. Even when $\tilde R_\delta = R_\delta$, the
%elliptic representation indexed by $r$ depends a choice of a normalization of the 
%intertwining operators (so on the choice of $\tilde J_\delta$).}

%Further, $I_P^G (\delta)_{r}$ depends only on the conjugacy class of $r$ in $R_\delta$. 

\begin{definition}
We set $R_{\delta,ell}:= R_\delta \cap W_{L,ell}$. Virtual representations  $I_P^G (\delta)_r$ with $r \in R_{\delta ,ell}$ are called elliptic tempered $G$-representations. More generally, virtual representations of the form  $I_P^G (\delta)_r\otimes\chi$ with  $r \in R_{\delta ,ell}$ and $\chi\in X_{\nr}(G)$ are called elliptic $G$-representations.
\end{definition}

By definition, every square-integrable (modulo centre) $G$-representation is elliptic
tempered. One can characterize characters of elliptic tempered $G$-representations also as characters of $G$ 
that decrease rapidly (in a precise sense) when $g \in G$ goes to infinity \cite{Herb}.

%\begin{lemma}\label{lem:5.1} \textup{\cite{Herb}} \\
%The virtual representations $I_P^G (\delta)_r$, where $r$ runs through the conjugacy classes of elliptic elements in
%$R_\delta$, are linearly independent.
%\end{lemma}

\begin{definition} 
\begin{enumerate}
\item For every irreducible
square-integrable (modulo centre) representation $\delta$ of a Levi subgroup of $G$ we denote by 
$\mcK_{G,\delta} (ar) \subset \mcK_{G,\text{temp}} $ the $\OQ$-span of the
virtual representations $I_P^G (\delta )_r$ with $r \in R_{\delta,ell}$. 

\item We denote by $\mcK_{G,\text{temp}} (ar) \subset \mcK_{G,\text{temp}}$ the span of all 
the the subspaces $\mcK_{G,\delta}$, and denote by $\mcK_{G}(ar) \subset \mcK_{G}$ the $\OQ$-span of the virtual representations $\pi\otimes\chi$ with $\pi\in\mcK_{G,\text{temp}} (ar)$ and $\chi\in X_{nr}(G)$. 

\item  We set $\mcK'_{G}(ar):=\mcK_{G}(ar)\cap \mcK'_{G}$, $\mcK'_{G,\text{temp}}(ar):=\mcK_{G,\text{temp}}(ar)\cap \mcK'_{G}$, $\mcK_{\OQ G}(ar):=\mcK_G (ar)\cap \mcK_{\OQ G}$, $\mcK'_{G,\text{temp}}(ar):=\mcK_{G,\text{temp}}(ar)\cap \mcK'_{G}$, and $\mcK_{\OQ G,\text{temp}}(ar):=\mcK_{G,\text{temp}} (ar)\cap \mcK_{\OQ G}$.
\end{enumerate}
\end{definition}

The following two results are not needed in this paper, but they highlight the
importance of the space of elliptic (tempered) $G$-representations.
For a Levi subgroup $L$ of $G$ we denote by $\mcK_G^L (ar) \subset \mcK_G$
the image of the space $\mcK_L (ar)$ under the map $I_P^G : \mcK_L \to \mcK_G $. 
This subspace depends only on the conjugacy class of $L$. 

\begin{theorem}\label{thm:5.6} \textup{\cite{Ar}} \\
$\mcK_{G,\text{temp}} = \bigoplus_L \mcK_{G,\text{temp}}^L (ar)$ where $L$ runs through 
the set of conjugacy classes of Levi subgroups of $G$ (including $G$ itself).
\end{theorem}

\begin{corollary} 
$\mcK_G$ equals $\bigoplus\nolimits_L \mcK_G^L (ar)$.
\end{corollary}
\begin{proof}
This follows from Theorem \ref{thm:5.6} and the comparison of the geometric structures
of $\Irr (\C G)$ and its subspace of tempered representations, as in 
\cite[Proposition 2.1]{ABPS}.
\end{proof}

\subsection{Realization of elliptic representations over $\OQ$} \label{par:7.3} \

As before, let $L\subset G$ be a Levi subgroup. The following simple observation will be central for what follows.  

\begin{lemma}\label{lem:5.3}
Let $\delta \in \Irr (\C L)$.  
\begin{enumerate}
\item Assume that $W_{L,\delta,ell}$ is nonempty. There exists 
$\chi_G \in X_\nr (G)$ such that the image of the central character of $\delta \otimes (\chi_G|_L)$ is finite.

\item The set $\{ \chi \in X_\nr (L) : W_{L,\delta \otimes \chi,ell} \text{ is nonempty} \}$
is a finite union of $X_\nr (G)$-orbits. 
\end{enumerate}
\end{lemma}
\begin{proof}
(1) Arguing as in Lemma~\ref{lem:finim}, we conclude that there exists 
$\chi_G \in X_\nr (G)$ such that the image of the tensor product cc$(\delta)|_{Z(G)}\otimes \chi_G|_{Z(G)}$ is finite. We claim that this $\chi_G$ satisfies the required property. Indeed, replacing $\delta$ 
by $\delta \otimes (\chi_G|_L)$ we can assume that $w(\delta)\cong \delta$ for some $w\in W_{L,ell}$ and the image cc$(\delta)(Z(G))$ is finite. We want to show that the image of $\mu:=$cc$(\delta):Z(L)\to\C^{\times}$ is finite.

Choose an uniformizer $\varpi\in F$. Then the map $\nu\mapsto\nu(\varpi)$ gives an injective map 
$X_*(Z(L))\hookrightarrow Z(L)$ such that $X_*(Z(L))Z(L)^1\subset Z(L)$ is a subgroup of finite index. Since $Z(L)^1$ is compact and $\mu$ is smooth, the image $\mu(Z(L)^1)$ is finite. 
Therefore it suffices to show that the image  $\mu(X_*(Z(L)))$ is finite. 

Since $w(\delta)\cong\delta$, we conclude that $w(\mu)=\mu$, thus $\mu$ is tvivial on the image $\text{Im}(w-1)\subset X_*(Z(L))$. On the other hand, the assumption that $w$ is elliptic implies that $X_*(Z(G))+\text{Im}(w-1)\subset X_*(Z(L))$ is a subgroup of finite index. 
Therefore finiteness of $\mu(X_*(Z(L)))$ follows from the finiteness of $\mu(Z(G))$.  

(2) Since $W_L$ is finite, it suffices to show that for every $w \in W_{L,ell}$ the set 
$\{ \chi \in X_\nr (L) : w \in W_{L,\delta \otimes \chi} \}$ is a finite union of $X_\nr (G)$-orbits. Replacing $\delta$ by $\delta\otimes\chi$, if necessary,
we may assume that $w(\delta)=\delta$. Let $X_\nr (L,\delta)$ be the collection of all  $\chi\in X_\nr (L)$
such that $\chi\otimes\delta\simeq\delta$, so that there is a bijection
\[
X_\nr (L) / X_\nr (L,\delta) \to X_\nr (L) \delta : \chi \mapsto \delta \otimes \chi .
\]
This map is $w$-equivariant, so it induces a bijection between the $w$-fixed points on
both sides. Since $w$ is assumed to be elliptic, $(X_\nr (L)/X_\nr (L,\delta))^w$
is a finite union of $X_\nr (G)$-orbits. Hence $(X_\nr (L) \delta )^w$ is a finite union of 
$X_\nr (G)$-orbits as well. 
\end{proof}

Consider a square-integrable (modulo centre) $\delta \in \Irr (\C L)$.
%and $\chi \in X_\nr (L)$. Assume for the moment that $\chi$ is unitary, so that the involved
%representations are tempered and \cite{Ar} applies.
By definition, the space of elliptic representations $\mcK_{G,\delta}$ is nonzero
if and only if $R_{\delta,ell}$ is nonempty, and this condition implies that 
$W_{L,\delta,ell}$ is nonempty. 
%contains elliptic elements is necessary for
%$R_{\delta,ell}$ to be nonempty, but not sufficient. Fortunately, it already gives us
%enough structure to derive some nice properties of representations. 

\begin{theorem}\label{thm:5.4}
Let $\delta \in \Irr (\C L)$ be a square-integrable representation (modulo centre) such that 
$R_{\delta,ell}$ is nonempty. The $G$-representations $I_P^G (\delta)_\kappa$ with 
$\kappa \in \Pi (\delta, \zeta)$ can be realized over $\OQ$ if and only if 
cc$(I_P^G (\delta)_\kappa) = \text{cc}(\delta) |_{Z(G)}$ takes values in $\OQ^\times$. 
In that case, $\delta$ can be realized over $\OQ$, all elliptic representations  
$I_P^G (\delta )_r$ with $r \in R_{\delta,ell}$ are defined over $\OQ$,  
and $\mcK_{G,\delta}(ar) = \mcK_{\OQ G,\delta}(ar)$.
\end{theorem}
\begin{proof}
$\Rightarrow$ If each $I_P^G (\delta)_\kappa$ can be realized over $\OQ$, then by Proposition 
\ref{prop:3.14} their cuspidal support Sc$(I_P^G (\delta)_\kappa) = \text{Sc}(\delta)$ 
can be realized over $\OQ$. Another application of Proposition \ref{prop:3.14} shows that 
$\delta$ can be realized over $\OQ$. Hence cc$(\delta)$ takes values in $\OQ^\times$.\\
$\Leftarrow$ By Theorem \ref{thm:4.7}.(2) there exists $\chi\in X_\nr (L)$ such
that $\delta' := \delta\otimes \chi^{-1}$ can be realized over $\OQ$. We claim
that $\chi$ takes values in $\OQ^\times$.

Assuming this claim for the moment, $\delta = \delta'\otimes \chi$ can be realized 
over $\OQ$, so $I_P^G (\delta)$ can be realized 
over $\OQ$ as well. In that case Lemma \ref{lem:3.12} shows that 
all its subrepresentations $I_P^G (\delta)_\kappa$  can be realized over $\OQ$.

The proof of the claim is almost identical to that of Theorem \ref{thm:4.7}.(1). Arguing as 
over there, we deduce from Lemma \ref{lem:5.3}.(2) that the restriction $\chi |_{L\cap G^1}$ 
has values in $\OQ^\times$. Next, since $cc(\delta)|_{Z(G)}$ has values in $\OQ^\times$ and 
$\delta'$ can be realized over $\OQ$, we deduce that $\chi |_{Z(G)}$ takes values in 
$\OQ^\times$. Hence $\chi |_{L\cap G^1Z(G)}$ takes values in $\OQ^\times$, therefore $\chi$ 
also has this property because $[G : G^1 Z(G)] < \infty$.

When the above equivalent conditions are fulfilled, the elliptic representation 
$I_P^G (\delta )_r$ with $r \in R_{\delta,ell}$ are defined over $\OQ$. 
%When $r$ is considered up to conjugacy, Lemma \ref{lem:5.1} says that 
By definition, these functions span over $\OQ$ both $\mcK_{G,\delta}(ar)$ and $\mcK_{\OQ G,\delta}(ar)$.
\end{proof}

We are ready to conclude the proof of Theorem \ref{T2}.

\begin{corollary} \label{cor:5.7}
\begin{enumerate}
\item Every elliptic representation in $\mcK'_G (ar)$ can be realized over $\OQ$.
\item $\mcK'_{G,\text{temp}} (ar)$ equals $\mcK_{\OQ G,\text{temp}}(ar)$ and
$\mcK'_G (ar)$ equals $\mcK_{\OQ G} (ar)$.
\end{enumerate}
\end{corollary}
\begin{proof}
(1) By Lemma~\ref{lem:5.3}.(1), the space $\mcK'_G (ar)$ is spanned by virtual
repre\-sen\-ta\-tions $\pi=I_P^G (\delta )_r\otimes\chi$, where $I_P^G (\delta )_r$ 
is a tempered elliptic $G$-repre\-sen\-ta\-tion, corresponding to a square-integrable 
representation $\delta$, whose central character is of finite order, and 
$\chi\in X_{nr}(G)$ is a $\OQ^{\times}$-valued unramified character of $G$.  
Since the virtual representation $I_P^G (\delta')_\kappa$ can be realized over $\OQ$ 
by Theorem~\ref{thm:5.4}, $\pi$ can be realized over $\OQ$ as well. 

%By definition, $\mcK'_G (ar)$ is spanned over $\OQ$ by virtual $G$-representations 
%$\pi=I_P^G (\delta)_\kappa\otimes\chi$, where $\delta \in \Irr (\C L)$ is a square-integrable representation (modulo centre) such that $R_{\delta,ell}$ is nonempty, $\chi\in X_{\nr}(G)$ and 
%the central character cc$(\pi)=\delta|_{Z(G)}\otimes\chi|_{Z(G)}$ takes values in $\OQ^{\times}$. 

%Since  $R_{\delta,ell}$ is nonempty, we conclude from Lemma~\ref{lem:5.3}.(1) that there exists 
%$\chi'\in X_{\nr}(G)$ such that $\delta'=\delta\otimes\chi'$ has the property 
%such that $\chi'\in X_{\nr}(G)$ and the image of cc$(\delta')$ is a finite subgroup 
%of $\{ z \in \OQ^\times : |z| = 1\}$. As in Theorem~\ref{thm:4.7}, representation $\delta'$ is square-integrable (modulo centre) and can be realized over $\OQ$.
%Then $\pi\cong I_P^G (\delta')_\kappa\otimes\chi''$ with $\chi''=\chi\otimes\chi'$, and virtual 
%representation $I_P^G (\delta')_\kappa$ can be realized over $\OQ$ by Theorem~\ref{thm:5.4}. 
%Now the assumption that cc$(\pi)$ takes values in $\OQ^{\times}$ implies that   $\chi''$ 
%takes values in $\OQ^{\times}$, thus $\pi$ can be realized over $\OQ$.

(2) This is direct consequence of part (1).
\end{proof}
%Theorem \ref{thm:5.4} shows that an elliptic $\C G$-representation $I_P^G (\delta )_r$ 
%can be rea\-lized over $\OQ$ if and only if its central character 
%cc$(\delta) |_{Z(G)}$ takes values in~$\OQ^\times$.

\subsection{The action of Gal$(\C / \mQ )$} \label{par:7.4} \

We consider the Galois action on elliptic representations, generalizing the
Galois action on essentially square-integrable representations studied in
Section \ref{sec:5}.
We have to take into account that $I_P^G$ need not commute with the action of an 
arbitrary $\gamma \in \Gal (\C / \mQ)$. %Let $\sqrt{\delta_P}$ and $\sqrt[']{\delta_P} = 
%\sqrt{\delta_P} \chi_-$ be the square roots of $\delta_P : P \to q_F^\ZZ$.
To this end, we consider the alternative square root of $\delta_P$. Recall from
\cite[\S 1.2.1]{Sil0} that
\[
\delta_P (l) = | \det \big( \text{Ad}(l) : \text{Lie}(U_P) \to \text{Lie}(U_P) \big) \big|_F .
\]
This shows in particular that $\delta_P$ takes values in $q_F^\ZZ \subset \mQ_{>0}$. 
Whenever $\delta_P (l)$ is a square
in $\mQ^\times$, it has a canonical root in $\mQ$, namely $\sqrt{\delta_P (l)} \in 
\mQ_{>0}$. When $\delta_P (l)$ is not a square in $\mQ^\times$, it has the roots
$\sqrt{\delta_P (l)} \in \RR_{>0}$ and $\sqrt[']{\delta_P (l)} \in \RR_{<0}$, which
are exchanged by $\Gal (\mQ (\sqrt{q_F}) / \mQ)$.
This gives rise to two roots of $\delta_P$, denoted $\sqrt{\delta_P}$ and 
$\sqrt[']{\delta_P}$. Let $v_p$ be the $p$-adic valuation on $\mQ$, then
\begin{equation}\label{eq:6.35}
\begin{aligned}
& \sqrt[']{\delta_P} = \sqrt{\delta_P} \chi_- , \\
& \chi_- (l) = (-1)^{v_p (\delta_P (l))} \qquad l \in L .
\end{aligned}
\end{equation}
Notice that $\chi_-$ is a quadratic unramified character of $L$, which could be 1. 

\begin{lemma}\label{lem:5.2}
\begin{enumerate}
\item $\chi_- |_L$ depends only on $L$, not on $P$.
\item $\chi_- |_L$ is fixed by $N_G (L)$.
\end{enumerate}
\end{lemma}
\begin{proof}
(1) Let $L_\alpha$ be the Levi subgroup of $G$ generated by $L \cup U_\alpha \cup 
U_{-\alpha}$, where $U_\alpha$ denotes the root subgroup for all positive multiples of
$\alpha$. Multiplication provides an isomorphism    
\[
\prod\nolimits_{\alpha \in \Phi (G,Z(L))^+} (U_P \cap L_\alpha) \to U_P,
\]
for any ordering of the set of roots $\Phi (G,Z(L))^+$ \cite[Proposition 14.4]{Bor}. 
This implies that
\[
\delta_P (l) = \prod\nolimits_{\alpha \in \Phi (G,Z(L))^+} \delta_{P \cap L_\alpha}(l)
\qquad l \in L .
\]
Writing $\chi_-^\alpha$ for $\chi_-$ with respect to $L_\alpha \supset P \cap L_\alpha
= L U_\alpha \supset L$, we obtain
\begin{equation}\label{eq:5.2}
\chi_- |_L = \prod\nolimits_{\alpha \in \Phi (G,Z(L))^+} \chi_-^\alpha |_L .    
\end{equation}
As $\delta_{L U_{-\alpha}} = \delta_{L U_\alpha}^{-1} = \delta_{P \cap L_\alpha}^{-1}$
and $\chi_-^\alpha$ is quadratic, we can also interpret $\chi_-^\alpha$ as $\chi_-$ for
$L_\alpha \supset L U_{-\alpha} \supset L$. This enables us to rewrite \eqref{eq:5.2} as
\[
\chi_- |_L = \prod\nolimits_{\alpha \in \Phi (G,Z(L)) / \{\pm 1\}} \chi_-^\alpha |_L .    
\]
That expression does not depend on the choice of a parabolic subgroup with Levi
factor $L$.\\
(2) For $n \in N_G (L)$ and $l \in L$ we have $\delta_P (n^{-1} l n) = 
\delta_{n P n^{-1}}(l)$. Hence $\chi_- (n^{-1} l n) = \chi_- (l)$, where the second
$\chi_-$ is for $G \supset n P n^{-1} \supset L$. By part (1), $\chi_-$ coincides
with the original $\chi_-$.
\end{proof}

%As we saw in Paragraph \ref{par:6.1}, to define $I_P^G$ and $j_{G,L}$ we need a square root
%of $q_F$. Therefore we also need $\sqrt{q_F}$ to define R-groups.
%Recall from \eqref{eq:6.27} that $I_P^G (\pi)^\gamma$ is isomorphic to 
%$I_P^G \big(\pi^\gamma \otimes \chi_-^{\epsilon (\gamma)} \big)$.
For $\gamma \in \Gal (\C / \mQ)$ we write $\epsilon (\gamma) = 0$
if $\gamma$ fixes $\sqrt{q_F}$, and $\epsilon (\gamma) = 1$ otherwise.
This definition is designed so that 
$\sqrt{\delta_P}^\gamma = \sqrt{\delta_P} \otimes \chi_-^{\epsilon (\gamma)}$.

For $\pi \in \Rep (\C L)$ and $\gamma \in \Gal (\C / \mQ)$ we have
\begin{equation}\label{eq:6.27}
\begin{aligned}
I_P^G (\pi)^\gamma & = \mathrm{ind}_P^G \big( \pi \otimes \sqrt{\delta_P} \big)^\gamma 
\cong\mathrm{ind}_P^G \big( \pi^\gamma \otimes \sqrt{\delta_P}^\gamma \big) \\
& = \mathrm{ind}_P^G \big( \pi^\gamma \otimes \sqrt{\delta_P} \otimes 
\chi_-^{\epsilon (\gamma)} \big) = 
I_P^G \big(\pi^\gamma \otimes \chi_-^{\epsilon (\gamma)} \big) .
\end{aligned}
\end{equation}

\begin{proposition}\label{prop:5.5}
Let $\delta \in \Irr (\C L)$ be a square-integrable representation (modulo centre) such 
that $R_{\delta,ell}$ is nonempty and central character $cc(\delta)$ is of finite order.  
Let $r \in R_{\delta,ell}$ and let $\gamma \in \Gal (\C / \mQ )$. Then 
\begin{enumerate}
\item $R_{\delta^\gamma \otimes \chi_-^{\epsilon (\gamma)}} = R_\delta\subset W_L$ and 
$R_{\delta^\gamma  \otimes \chi_-^{\epsilon (\gamma)},ell} = R_{\delta,ell}$.
\item Under identification  $I_P^G (\delta)^\gamma \cong I_P^G \big(\delta^\gamma \otimes \chi_-^{\epsilon (\gamma)} \big)$ of 
(\ref{eq:6.27}) and part (1), the projective representations  $(J_\delta )^\gamma$ and $J_{\delta^\gamma  \otimes \chi_-^{\epsilon (\gamma)}}$ are isomorphic.  
\item $I_P^G (\delta)_r^\gamma$ is isomorphic to the elliptic representation 
$I_P^G \big( \delta^\gamma  \otimes \chi_-^{\epsilon (\gamma)} \big)_r$.
\end{enumerate}
\end{proposition}
\begin{proof}
(1) By Theorem \ref{Cl}, $\delta^\gamma \in \Irr (\C L)$ is again square-integrable (modulo centre), whose central character is of finite order.  By Lemma \ref{lem:5.2}.(2), $W_{L, \delta^\gamma \otimes 
\chi_-^{\epsilon (\gamma)}}$ equals $W_{L,\delta^\gamma}$. The group 
$W_{L,\delta^\gamma}$ equals $W_{L,\delta}$ because the actions of $N_G (L)$ and
$\Gal (\C / \mQ)$ on $\Irr (L)$ commute. It follows easily from \eqref{eq:6.27}, algebraic description of intertwining operators in \cite[proof of Th\'eor\`eme IV.1.1]{Wal}, 
and the definition of $j_{L_\alpha,L}$ that
\begin{equation}\label{eq:5.8}
\gamma (j_{L_\alpha,L} (\delta)) = j_{L\alpha,L} \big( 
\delta^\gamma \otimes \chi_-^{\epsilon (\gamma)} \big) .
\end{equation}
Therefore the set $\Phi_\delta$ of roots $\alpha$ such that $\mu_{L_\alpha,L} (\delta) = 0$ 
does not change upon applying $\gamma$ to $I_P^G (\delta)$. Now \eqref{eq:5.5} shows 
that the groups $W \big( \Phi_{\delta^\gamma \otimes \chi_-^{\epsilon (\gamma)}} \big)$ 
and $R_{\delta^\gamma \otimes \chi_-^{\epsilon (\gamma)}}$ are the same for 
$\delta^\gamma \otimes \chi_-^{\epsilon (\gamma)}$ and for $\delta$. 
By \eqref{eq:5.4}, the same holds for the subsets of elliptic elements.

(2) For every $r \in R_\delta$, the intertwining operator $J_\delta (r) \in
\Aut_G (I_P^G (\delta))$ (defined up to a scalar) is characterized by the fact that it 
is a member of an algebraic family
\[
J_{\delta \otimes \chi}(r) \in \Hom_G \big( I_P^G (\delta \otimes \chi),
I_P^G (r \cdot (\delta \otimes \chi)) \big) \qquad \chi \in X_\nr (L) .
\]
Indeed, for generic $\chi$, $I_P^G (\delta \otimes \chi)$ is irreducible and
$J_{\delta \otimes \chi}(r)$ is unique up to a scalar. For any $\chi \in X_\nr (L)$,
the operator $J_{\delta \otimes \chi}(r)^{\gamma}$ is an element of 
\begin{multline*}
\Hom_G \big( I_P^G (\delta \otimes \chi)^\gamma, I_P^G (r \cdot 
(\delta \otimes \chi ))^\gamma \big) = \\
\Hom_G \big( I_P^G (\delta^\gamma \otimes \chi^\gamma \otimes \chi_-^{\epsilon (\gamma)}),
I_P^G (r \cdot (\delta^\gamma \otimes \chi^\gamma) \otimes \chi_-^{\epsilon (\gamma)}) \big). 
\end{multline*}
As $r \cdot \chi_- = \chi_-$ (Lemma \ref{lem:5.2}.(2)), a family of such operators  
determines $J_{\delta^\gamma \otimes \chi_-^{\epsilon (\gamma)}}(r)$.  
Therefore operators $J_\delta(r)^{\gamma}$ and $J_{\delta^\gamma \otimes \chi_-^{\epsilon (\gamma)}}(r)$ 
are equal up to a scalar.

(3) As in paragraph \ref{par:7.2}, we choose a central extension (\ref{extension}) and a lift $\ti J_{\delta}$ of a projective representation 
$\tau\circ J_{\delta}$. Then  $(\ti J_{\delta})^{\gamma}$ is a lift of $\tau\circ (J_{\delta})^{\gamma}$. Therefore, by part (2), 
$(\ti J_\delta )^\gamma$ can be viewed as a lift of 
$\tau\circ J_{\delta^\gamma \otimes \chi_-^{\epsilon (\gamma)}}$. Then, for every $\kappa \in \Pi (\delta,\zeta)$, we have $\kappa^\gamma \in \Pi (\delta^{\gamma},\zeta^{\gamma})\cong\Pi (\delta^\gamma \otimes \chi_-^{\epsilon (\gamma)},\zeta^\gamma)$ and   
\begin{equation}\label{eq:5.1}
I_P^G (\delta )_\kappa^\gamma\cong 
I_P^G (\delta^\gamma \otimes \chi_-^{\epsilon (\gamma)})_{\kappa^\gamma} .
\end{equation}
Upon applying $\gamma$ and using \eqref{eq:5.1}, the decomposition \eqref{eq:5.6} becomes
\begin{equation}\label{eq:5.10}
\begin{aligned}
I_P^G \big( \delta^\gamma \otimes \chi_-^{\epsilon (\gamma)} \big) & \cong 
I_P^G (\delta )^\gamma \cong \bigoplus\nolimits_{\kappa \in \Pi (\delta,\zeta)}
\kappa^\gamma \otimes I_P^G (\delta)_\kappa^\gamma \\
& = \bigoplus\nolimits_{\kappa^\gamma \in \Pi (\delta^\gamma \otimes \chi_-^{\epsilon (\gamma)},
\zeta^\gamma)} \kappa^\gamma \otimes 
I_P^G \big( \delta^\gamma \otimes \chi_-^{\epsilon (\gamma)} \big)_{\kappa^\gamma} .
\end{aligned}
\end{equation}
It follows that \eqref{eq:5.7} transforms into equality 
\[
I_P^G (\delta)_r^\gamma = \sum_{\kappa^\gamma \in \Pi (\delta^\gamma \otimes 
\chi_-^{\epsilon (\gamma)},\zeta^\gamma)} \hspace{-5mm} \text{tr}(\kappa^\gamma (\tilde r)) 
\, I_P^G (\delta^\gamma \otimes \chi_-^{\epsilon (\gamma)})_{\kappa^\gamma} = 
I_P^G (\delta^\gamma \otimes \chi_-^{\epsilon (\gamma)} )_r . \qedhere
\]
\end{proof}

We are ready to conclude the proof of Theorem \ref{thm:1.3}.

\begin{corollary}\label{thm:5.7}
\begin{enumerate}
\item For every elliptic $G$-representation $\pi$ and every $\Gg\in\Gal (\C / \mQ)$, the Galois conjugate 
$\pi^{\Gg}$ is elliptic. 
\item The spaces  $\mcK_G (ar)$ and  $\mcK'_G (ar) = \mcK_{\OQ G} (ar)$ are stable under $\Gal (\C / \mQ)$.
%\item $\nu (\mcK'_G (ar)) = \nu (\mcK_{\OQ G} (ar))$ is contained in $\OQ (G^{rs})$
%and stable under $\Gal (\C/ \mQ)$.
\end{enumerate}
\end{corollary}
\begin{proof}
(1) By Lemma~\ref{lem:5.3}.(1), every elliptic $G$-representation $\pi$ is of the form 
$I_P^G (\delta )_r\otimes\chi$, where $I_P^G (\delta )_r$ is a tempered elliptic $G$-representation, corresponding to a square-integrable representation $\delta$, whose central character is of finite order, and $\chi\in X_{\nr}(G)$ is a $\OQ^{\times}$-valued unramified character of $G$.  Since by Proposition \ref{prop:5.5}.(3), $\Gal (\C / \mQ)$ maps the $I_P^G (\delta )_r$'s to elliptic representations, the assertion follows. 

(2) This is a direct consequence of part (1).
\end{proof}
%(4) The equality follows from part (2) and the stability from part (3). 
%The $\OQ$-vector space $\nu (\mcK_{\OQ G} (ar))$ 
%is spanned by the Harish-Chandra characters of elliptic re\-pre\-sentations of $G$ 
%that can be realized over $\OQ$. The definition of Harish-Chandra 
%characters (see Theorem \ref{L1}) shows that these take values in $\OQ$.

\vspace{3mm}
\noindent\textbf{Acknowledgements.}\\
We thank Laurent Clozel for an explanation of his proof of Theorem \ref{Cl} 
and Eric Opdam for helpful explanations about formal degrees.  We thank 
Jean-Loup Waldspurger and Volker Heiermann for correcting imprecisions and providing references. We thank Jean-Francois Dat for pointing out that some results in an 
earlier version of this paper had already been proven by Vign\'eras in \cite{Vig}.\\
The research of D. Kazhdan was partially supported by ERC grant 101142781.
The research of Y. Varshavsky was partially supported by the ISF grant 2091/21.

\end{document}